\documentclass[12pt]{article}

\usepackage{times}
\usepackage{amsthm}
\usepackage{amsmath}
\usepackage{amssymb}
\usepackage{mathrsfs}
\usepackage{pb-diagram}
\usepackage{color}
\usepackage{extarrows}

\usepackage{hyperref}

\renewcommand\eqref[1]{(\ref{#1})} 

\def\H{\mathcal H}

\def\R{\mathbb{R}}

\def\r{{\rm r}}
\def\d{{\rm d}}

\def\({\left(}
\def\[{\left[}
\def\){\right)}
\def\]{\right]}

\def\si{\sigma}
\def\Si{\Sigma}

\def\G{{\sf G}}

\def\p{\parallel}
\def\<{\langle}
\def\>{\rangle}

\def\supp{\mathop{\mathrm{supp}}\nolimits}

\newtheorem{Theorem}{Theorem}[section]
\newtheorem{Remark}[Theorem]{Remark}
\newtheorem{Lemma}[Theorem]{Lemma}

\newtheorem{Proposition}[Theorem]{Proposition}
\newtheorem{Definition}[Theorem]{Definition}

\newtheorem{Hypothesis}[Theorem]{Hypothesis}

\numberwithin{equation}{section}

\begin{document}


\title{Localization and Non-propagation\\ in a Groupoid Framework}

\date{\today}

\author{M. M\u antoiu 
}
\date{\small}
\maketitle \vspace{-1cm}


\begin{abstract}
Normal elements (or multipliers) of the $C^*$algebra of a certain class of locally compact groupoids admit a natural faithful representation as normal operators on the $L^2$-space of a dense orbit of the groupoid. We prove norm estimates on the product between elements of the functional calculus of these operators and multiplication operators, subject to suitable restrictions expressed in terms of the orbit structure of the groupoid. As a consequence, one gets uniform estimates on the local behavior of the evolution group of the operators.
\end{abstract}

\section{Introduction}\label{introduction}

Let $H$ be a bounded normal operator in an $L^2$-Hilbert space $\H:=L^2(M;\mu)$\,. For any continuous real function $\kappa$ we denote by $\kappa(H)$ the normal operator in $\H$ constructed via the functional calculus. If $\Psi:M\to\R$ is a bounded measurable function, we denote by the same symbol the operator of multiplication by $\Psi$ in $\H$\,. Setting $\mathbb B(\H)$ for the space of all bounded linear operators on $\H$ and ${\sf sp}(H)$ for the spectrum of $H$, there is a single obvious inequality
\begin{equation}\label{apriori}
\p\!\Psi\kappa(H)\!\p_{\mathbb B(\H)}\,\le\,\sup_{x\in M}|\Psi(x)|\sup_{\lambda\in{\sf sp}(H)}\!|\kappa(\lambda)|
\end{equation}
that holds without extra assumptions. 

\smallskip
We intend to treat situations in which one can conclude that the left hand side is small
\begin{equation}\label{aposteriori}
\p\!\Psi\kappa(H)\!\p_{\mathbb B(\H)}\,\le\,\epsilon
\end{equation}
for some arbitrary $\epsilon>0$ given in advance, although in the right hand side of \eqref{apriori} the two factors are (say) equal to $1$\,. 

\smallskip
To see the usefulness of such a result, suppose that $H$ is self-adjoint and one is interested in its unitary group $\big\{e^{itH}\!\mid\!t\in\R\big\}$\,, perhaps describing the evolution of a quantum system. Then one gets immediately
\begin{equation}\label{aporiori}
\p\!\Psi e^{itH}\kappa(H)u\!\p_\H\,\le\varepsilon\!\p\!u\!\p_\H\,,
\end{equation}
and this holds for every "time" $t\in\R$ and any "state" $u\in\H$\,. We used terminology from Quantum Mechanics, having in mind the quantum interpretation of quantities as $\p\!\Psi v_t\!\p^2_\H$ (a localization probability). Actually, in our case, characteristic functions $\Psi=\mathbf 1_U$ are available and, if $H$ is the Hamiltonian of a quantum system, a condition of the form $\kappa(H)u=u$ means roughly that the vector (quantum state) $u$ have energies belonging to the support of the function $\kappa$\,, interpreted as a "localization in energy". In a typical case, $\kappa$ could be a good approximation of the characteristic function of a real interval $I$. So, finally, we expect \eqref{aporiori} to hold because of some correlation between the energy interval $I$ and the region $U\subset M$ in which "propagation is improbable for time-dependent states $u_t=e^{itH}u$ with $u$ having energies only belonging to $I$". 

\smallskip
Results as those described above have been obtained in \cite{DS} (a very particular case) and in \cite{AMP,MPR}, for operators $H$ deduced in a natural way from some dynamical system defined by the action of an Abelian locally compact group $\G$ in compactifications of $\G$\,. These dynamical systems are taken into account through the crossed product $C^*$-algebras they define \cite{Wi}. The present paper investigates the problem in the framework of the much more general groupoid $C^*$-algebras \cite{Pa,Re}. We work with amenable Hausdorff locally compact groupoids $\Xi$\,, possessing a Haar system  and with their associated full groupoid $C^*$-algebra ${\sf C}^*(\Xi)$ (coinciding with the reduced one). It is assumed that inside the unit space $X=\Xi^{(0)}$ there is a dense open orbit $M$ on which the isotropy groups are trivial. Under such assumptions, there is a natural measure $\mu$ on $M$ and a natural faithful representation $\Pi_0$ of ${\sf C}^*(\Xi)$ in the Hilbert space $L^2(M;\mu)$\,. In Section \ref{goam} the basic facts about groupoid algebras are reviewed.

\smallskip
The main abstract result is proven in Section \ref{hoameni}. The objects of our investigations are normal elements (or even multipliers) $F$ of the groupoid algebra and operators $H_0\!:=\Pi_0(F)$ acting in $L^2(M;\mu)$\,. An important role is played by "the region at infinity" $N:=X\!\setminus\!M$. For every quasi-orbit (closure of an orbit) $\mathcal Q$ contained in $N$ one associates an element $F_\mathcal Q$ in the $C^*$-algebra of the reduced groupoid $\Xi_\mathcal Q$\,, with spectrum contained in the essential spectrum of the operator $H_0$\,. The function $\kappa$ in \eqref{aposteriori} has to be supported away from ${\sf sp}(F_\mathcal Q)$\,.
Then the trace on $M$ of small neighborhoods $W$ of $\mathcal Q$ in the unit space $X$ define the functions $\Psi$ admitted in \eqref{aposteriori}. In terms of the evolution group, informally, one may say that "at energies not belonging to the asymptotic observable $F_\mathcal Q$\,, propagation towards the quasi-orbit $\mathcal Q$ is very improbable".

\smallskip
In the final section we outline a class of examples, not deriving from group actions, leading finally to a situation in which both the nature and spectrum of the asymptotic observable $F_\mathcal Q$ and the neighborhood of the corresponding quasi-orbit are transparent enough. The unit space $X$ is built as a compactification of an a priori given $M$, defined by a continuous surjection from the complement of a compact subset of $M$ to a compact space $N$. It is also assumed that the reduction of the groupoid to $N$ is a group bundle. Then the observable corresponding to the orbits in $\mathcal Q\subset N$ are group convolution operators, converted by a Fourier transformation in multiplication operators in the Abelian case.

\section{The framework}\label{goam}

We assume that the reader is familiar with the basic theory of locally groupoids and their $C^*$-algebras. In this section we only mention some elementary facts, mainly in order to fix notations, for which we refer to \cite{Pa,Re}. 

\smallskip
Let $\Xi$ be a locally compact groupoid over a compact unit space $\Xi^{(0)}\!\equiv X$, with fixed right Haar system $\{\lambda_x\!\mid\!x\in X\}$\,. Composing with inversion, one also gets the left Haar system $\{\lambda^x\!\mid\!x\in X\}$\,.The source and range maps are denoted by ${\rm d,r}:\Xi\to \Xi^{(0)}$\,. For $A,B\subset X$ one writes 
$$
\Xi_A:={\rm d}^{-1}(A)\,,\quad\Xi^B:={\rm r}^{-1}(B)\,,\quad\Xi_A^B:=\Xi_A\cap\Xi^B.
$$
For each unit $x\in X$ we set $\Xi_x\equiv\Xi_{\{x\}}$ and $\Xi^x\equiv\Xi^{\{x\}}$. The intersection $\Xi^x_x:=\Xi_x\cap\Xi_x$ is {\it the isotropy group} of the point $x$\,. We say that two units $x,y$ are {\it equivalent} (or that {\it they belong to the same orbit}) if there exists $\xi\in\Xi$ such that $x=\r(\xi)$ and $y={\rm d}(\xi)$\,. {\it The orbit of a point} $x$ will be denoted by $\mathcal O_x={\rm r}(\Xi_x)$\,. {\it A quasi-orbit} is the closure of an orbit. 

\begin{Definition}\label{admissible}
The locally compact groupoid $\Xi$ with right Haar system $\lambda$ is called {\rm standard} if it is amenable, Hausdorff, second-countable and it has a dense open orbit $M\subset X$ with trivial isotropy, i.\,e. $\Xi_z^z=\{z\}$ for (any) $z\in M$.
\end{Definition} 

If $(\Xi,\lambda)$ is given as above, the vector space $C_c(\Xi)$ of all continuous compactly supported complex functions on $\Xi$ becomes a $^*$-algebra with the composition law
$$
(f\star g)(\xi):=\int_{\Xi}f(\eta)g(\eta^{-1}\xi)\,d\lambda^{\r(\xi)}(\eta)=\int_{\Xi}f(\xi\eta)g(\eta^{-1})\,d\lambda^{\d(\xi)}(\eta)
$$
and the involution
$$
f^{\star}(\xi):=\overline{f(\xi^{-1})}\,.
$$
Then, as explained in \cite{Re}, one gets {\it the (full) groupoid $C^*$-algebra}. By amenability \cite{ADR}, this full $C^*$-algebra coincides with the reduced one; we use the common notation ${\sf C}^*(\Xi)$\,. Recall that for every $x\in X$ there is a {\it regular representation} $\Pi_x:{\sf C}^*(\Xi)\to\mathbb B\big[L^2(\Xi_x;\lambda_x)\big]$ defined through
$$
\Pi_x(f)u:=f\star u\,,\quad\forall\,f\in C_{\rm c}(\Xi)\,,\ u\in L^2(\Xi_x;\lambda_x)\equiv\H_x\,.
$$
It follows easily that for every $\xi\in\Xi$ one has the unitary equivalence $\,\Pi_{\r(\xi)}\approx\Pi_{\d(\xi)}$\,, so the regular representations along an orbit are all unitarily equivalent. If the orbit $\mathcal O_x$ of the unit $x$ is dense in $X$, the regular representation $\Pi_x$ is known to be faithful; for this see for example \cite[Prop.\,2.7]{BB}, relying on results from \cite{KS}. In our case this applies to any $x\in M$.

\smallskip
If $A$ is {\it invariant} (a union of orbits), $\Xi_A^A=\Xi_A=\Xi^A$ is a subgroupoid, called {\it the reduction of $\,\Xi$ to $A$}\,. If $A$ is also open or closed, then $A$ is also locally compact and on $\Xi_A$ we use the restriction of the Haar system. For the next result, basic for our approach, we refer to \cite[Lemma 2.10]{MRW3}.

\begin{Lemma}\label{apatra}
Let $B$ be an open $\Xi$-invariant subset of $X$. Then $A:=B^{\rm c}$ is closed and $\Xi$-invariant and one has  reductions $\Xi_B$ and $\Xi_{A}$ and  a short exact sequence of groupoid $C^*$-algebras
\begin{equation}\label{ses}
0\longrightarrow{\sf C}^*(\Xi_B)\overset{\iota_B}{\longrightarrow}{\sf C}^*(\Xi)\overset{\rho_A}{\longrightarrow}{\sf C}^*(\Xi_A)\longrightarrow 0\,.
\end{equation}
On continuous compactly supported functions the monomorphism $\iota_B$ consists of extending by the value $0$ and the epimorphism $\rho_A$ acts as a restriction.
\end{Lemma}

\begin{Remark}\label{lateriu}
{\rm In some situations, $C_{\rm c}(\Xi)$ is too small to contain enough elements of interest. On the other hand, the nature of the elements of the completion ${\sf C}^*(\Xi)$ is not easy to grasp; for instance it is not clear the concrete meaning of a restriction applied to them. In search of a good compromise, we recall \cite[pag.\,50]{Re} the embeddings
$$
C_{\rm c}(\Xi)\subset L^{\infty,1}(\Xi)\subset {\sf C}^*(\Xi)\,,
$$
where $L^{\infty,1}(\Xi)$ is a Banach $^*$-algebra, the completion of $C_{\rm c}(\Xi)$ in {\it the Hahn norm}
$$
\p\!f\!\p_{\Xi}^{\infty,1}:=\max\Big\{\sup_{x\in X}\int_{\Xi_x}\!|f(\xi)|d\lambda_x(\xi)\,,\sup_{x\in X}\int_{\Xi_x}\!|f\big(\xi^{-1}\big)|d\lambda_x(\xi)\Big\}\,.
$$
Actually, ${\sf C}^*(\Xi)$ is the enveloping $C^*$-algebra of $L^{\infty,1}(\Xi)$\,. The following result is obvious:

\smallskip
{\it If $A\subset X$ is closed and invariant, the restriction map $\rho_A$ extends to a $^*$-morphism $\rho_A:L^{\infty,1}(\Xi)\to L^{\infty,1}\big(\Xi_A\big)$ that is contractive with respect to the Hahn norms.}

\smallskip
It is still unclear how this extensions act on a general element $f\in L^{\infty,1}(\Xi)$\,. However, if $f$ is a continuous function on $\Xi$ with finite Hahn norm (write $f\in L^{\infty,1}_{\rm cont}(\Xi)$), the restrictions acts in the usual way and one gets an element $\rho_A(f)=f|_{\Xi_A}$ of $L^{\infty,1}_{\rm cont}(\Xi_A)$\,.}
\end{Remark}

\begin{Remark}\label{idem}
{\rm For each $z\in M$, the restriction of the range map
$$
\r_z\!:=\r|_{\Xi_z}\!:\Xi_z\to M
$$ 
is surjective (since $M$ is an orbit) and injective (since the isotropy is trivial). Thus one transports the measure $\lambda_z$ to a (full Radon) measure $\mu$ on $M$ (independent of  $z$, by the invariance of the Haar system) and gets a representation $\Pi_0:{\sf C}^*(\Xi)\to\mathbb B\big[L^2(M,\mu)\big]$\,, called {\it the vector representation}. Let 
$$
{\sf R}_z:L^2(M;\mu)\to L^2\big(\Xi_z;\lambda_z\big)\,,\quad {\sf R}_z(v):=v\circ r_z\,.
$$
Then $\Pi_0$ is defined to be unitarily equivalent to $\Pi_z$:
\begin{equation}\label{ibidem}
\Pi_0(f)={\sf R}_z^{-1}\Pi_z(f){\sf R}_z\,,\quad \Pi_0(f)v=\big[f\star(v\circ \r_z)\big]\circ \r_z^{-1}.
\end{equation}
Since the orbit $M$ is dense, $\Pi_z$ is faithful, so {\it the vector representation $\Pi_0$ is also faithful}.
In addition, $\Xi_M$ is isomorphic to the pair groupoid $M\times M$, the $C^*$-algebra ${\sf C}^*(\Xi_M)\cong{\sf C}^*(M\times M)$ is elementary and $\Pi_0\big[{\sf C}^*(\Xi_M)\big]=\mathbb K(\H_0):=$ the $C^*$-algebra of all compact operators on $\H_0:=L^2(M,\mu)$\,. 
}
\end{Remark}

\begin{Remark}\label{unitizari}
{\rm If $\mathscr C$ is a $C^*$-algebra, its multiplier $C^*$-algebra is denoted by $\mathscr C^\bullet$; it is the largest unital $C^*$algebras containing $\mathscr C$ as an essential ideal.  If $\mathscr C$ is already unital, one has $\mathscr C=\mathscr C^\bullet$. Any non-degenerate representation $\Pi:\mathscr C\to\mathbb B(\H)$ extends uniquely to a representation $\Pi^\bullet:\mathscr C^\bullet\to\mathbb B(\H)$\,. If $\Pi$ is injective, $\Pi^\bullet$ is also injective; this may be applied to the vector representation $\Pi_0$ introduced at Remark \ref{idem}. Recall that a surjective $^*$-morphism $\Phi:\mathscr C\to\mathscr D$ extends uniquely to a $^*$-morphism $\Phi^\bullet:\mathscr C^\bullet\to\mathscr D^\bullet$ between the multiplier algebras; if $\mathscr C$ and $\mathscr D$ are $\si$-unital (separable, in particular), then $\Phi^\bullet$ is also surjective. This will be applied to the restriction epimorphisms $\rho_A$\,, since the $C^*$-algebras attached to second countable groupoids are separable.}
\end{Remark}

\section{Localization and non-propagation properties}\label{hoameni}

We are given a normal element $F\in\,{\sf C}^*(\Xi)^\bullet$ and study its image $H_0:=\Pi_0^\bullet(F)$ in the vector representation; it is an operator in $\H_0:= L^2(M,\mu)$\,. For every closed invariant $B\subset X$ and for every $x\in X$, one sets 
$$
F_B:=\rho_B^\bullet(F)\in{\sf C}^*(\Xi_B)^\bullet\quad{\rm and}\quad H_x:=\Pi_x^\bullet(F)\in\mathbb B(\H_x)\,,
$$ 
where $\H_x:=L^2(\Xi_x;\lambda_x)$\,. By definition, the operator $H_0$ is defined to be unitarily equivalent to any $H_x$ with $x$ belonging to the open dense orbit $M$. Occasionally, although this is not strictly necessary, we will indicate the unital $C^*$-algebra in which the spectrum is computed. 

\smallskip
If $\varphi$ is an essentially bounded function on $M$, we denote by the same letter the corresponding operator of multiplication in $\H_0$\,. By $\textbf{1}_V$ we denote the charcteristic function of the set $V$, and the corresponding multiplication operator. 

\begin{Theorem}\label{main}
Let $F\in\,{\sf C}^*(\Xi)^\bullet$ be normal. Let us consider a quasi-orbit $\mathcal Q\subset N:=X\setminus M$. Let $\kappa\in C_0(\mathbb R)$ be a real function with support that does not meet ${\sf sp}\big(F_\mathcal Q\big)\equiv{\sf sp}\big(F_\mathcal Q\!\mid\!{\sf C}^*(\Xi_\mathcal Q)^\bullet\big)$\,. 

\begin{enumerate}
\item[(i)]
For every $\epsilon>0$ there is a neighborhood $W$ of $\mathcal Q$ in $X$ such that, setting $W_0:=W\cap M$, one has
\begin{equation*}\label{traznaie}
\p\!{\bf 1}_{W_0}\kappa(H_0)\!\p_{\mathbb B(\H_0)}\,\le\epsilon\,.
\end{equation*}
\item[(ii)]
Suppose that $F$ is self-adjoint. Uniformly in $t\in\R$ and $u\in\H_0$ one also has
\begin{equation}\label{traznea}
\p\!{\bf 1}_{W_0}e^{itH_0}\kappa(H_0)u\!\p_{\H_0}\,\le\epsilon\p\!u\!\p_{\H_0}.
\end{equation}
\end{enumerate}
\end{Theorem}

Of course (ii) follows from (i), because $e^{itH_0}$ is a unitary operator and it commutes with $\kappa(H_0)$\,. We put \eqref{traznea} into evidence because of its dynamical interpretation. So we need to show (i). 

\smallskip
We start with three preliminary results; by combining them, Theorem \ref{main} will follow easily. The first one is abstract, it deals with elements belonging to a $C^*$-algebra endowed with a distinguished ideal, and it reproduces \cite[Lemma 1]{AMP}.

\begin{Lemma}\label{lemma1}
Let $G$ be a normal element of a unital $C^*$-algebra $\mathscr A$ and let $\mathscr K$ be a closed self-adjoint two-sided ideal of $\,\mathscr A$. Denote by ${\sf sp}_\mathscr K(G)$ the spectrum of the canonical image of $G$ in the quotient $\mathscr A/\mathscr K$. If $\kappa\in C_0(\R)_\R$ and $\supp(\kappa)\cap{\sf sp}_\mathscr K(G)=\emptyset$\,, then $\kappa(G)\in \mathscr K$.
\end{Lemma}

The second preliminary result concerns a relationship between multiplication operators and elements of the twisted groupoid $C^*$-algebra.

\begin{Lemma}\label{lemma2}
\begin{enumerate}
\item[(a)]
The (Abelian) $C^*$-algebra $C(X)$ of complex continuous functions on $X$ acts by double centralizers
\begin{equation*}\label{socoteala}
C(X)\times C_{\rm c}(\Xi)\ni(\psi,f)\to\psi\cdot f:=(\psi\circ \r)\,f\in C_{\rm c}(\Xi)\,,
\end{equation*}
\begin{equation*}\label{zocoteala}
C_{\rm c}(\Xi)\times C(X)\ni(f,\psi)\to f\cdot\psi:=(\psi\circ \d)\,f\in C_{\rm c}(\Xi)\,.
\end{equation*}
This action extends to the groupoid $C^*$-algebra and leads to an embedding in the multiplier $C^*$-algebra: $C(X)\hookrightarrow {\sf C}^*(\Xi)^\bullet$\,.
\item[(b)]
The canonical extension $\Pi_0^\bullet$ of the vector representation to the multiplier algebra ${\sf C}^*(\Xi)^\bullet$ acts on $C(X)$ by multiplication operators:
\begin{equation}\label{titiriseala}
\Pi_0^\bullet(\psi)u:=\psi|_M\,u\,,\quad\forall \,\psi\in C(X)\,,\,u\in L^2(M)\,.
\end{equation}
\end{enumerate}
\end{Lemma}

\begin{proof}
The point (a) is straightforward; actually it is a particular case of \cite{Re}[II,Prop.\,2.4(ii)] (take there $\H:=X$ to be the closed subgroupoid of $\mathcal G=\Xi$)\,.

\smallskip
The point (b) is also straightforward. One has to check, for instance, the following identity in $\mathbb B\big[L^2(M)\big]$\,:
\begin{equation}\label{desguise}
\Pi_0(\psi\cdot f)=\psi|_M\,\Pi_0(f)\,,\quad\forall\,\psi\in C(X)\,,\,f\in C_{\rm c}(\Xi)\,.
\end{equation}
Recall that $\r_z$ is the restriction of range map $\r$ to the $\d$-fiber $\Xi_z$, and it is a bijection $\Xi_z\to M$ if $z\in M$. We can write for any $z\in M$, $\xi\in\Xi_z$ and $u\in L^2(M,\mu)\equiv L^2(M,\r_z(\lambda_z))$ 
$$
\begin{aligned}
\big[(\psi\cdot f)\star(u\circ \r_z)\big](\xi)&=\int_\Xi\psi[\r(\eta)]f(\eta)\,u[\r(\eta^{-1}\xi)]\,d\lambda^{\r(\xi)}(\eta)\\
&=\psi[\r(\xi)]\int_\Xi f(\eta)\,u[\r(\eta^{-1}\xi)]\,d\lambda^{\r(\xi)}(\eta)\\
&=\big[(\psi\circ \r_z)\big(f\star(u\circ \r_z)\big)\big](\xi)\,,
\end{aligned}
$$
which is \eqref{desguise} in desguise, by \eqref{ibidem}. We leave the remaining details to the reader.
\end{proof}

The third result essentially speaks of a bounded approximate unit.

\begin{Lemma}\label{lemma3}
Let $A\subset X$ be a closed invariant subset and $f\in{\sf C}^*\big(\Xi_{X\setminus A}\big)\subset{\sf C}^*(\Xi)$\,. For every $\epsilon>0$ there exists $\psi\in C(X)$ with $\psi(X)=[0,2]$\,, $\psi|_A=2$ and 
\begin{equation}\label{estimam}
\p\!\psi\cdot f\!\p_{{\sf C}^*(\Xi)}+\p\!f\cdot\psi\!\p_{{\sf C}^*(\Xi)}\,\le\epsilon\,.
\end{equation}
\end{Lemma}

\begin{proof}
Let us set $B:=X\setminus A$\,. By density,  there exists $f_0\in C_{\rm c}\big(\Xi_B\big)$ such that $\p\!f-f_0\!\p_{{\sf C}^*(\Xi)}\,\le\epsilon/4$\,. Set 
$$
S_0:=\d[\supp(f_0)]\cup \r[\supp(f_0)]\,;
$$ 
since $\Xi_B=\Xi^B\!=\Xi^B_B$\,, the subset $S_0$ is compact and disjoint from $A$\,. So there is a continuous function $\psi:X\to[0,2]$ with $\psi|_A=2$ and $\psi|_{S_0}=0$\,. In particular, $\psi\cdot f_0$ and $f_0\cdot\psi$ both vanish. Then
$$
\begin{aligned}
\p\!\psi\cdot f\!\p_{{\sf C}^*(\Xi)}\!+\p\!f\cdot\psi\!\p_{{\sf C}^*(\Xi)}&=\,\p\!\psi\cdot(f-f_0)\!\p_{{\sf C}^*(\Xi)}\!+\p\!(f-f_0)\cdot\psi\!\p_{{\sf C}^*(\Xi)}\\
&\le 2\p\!\psi\!\p_\infty\p\!f-f_0\!\p_{{\sf C}^*(\Xi)}\,\le\epsilon\,,
\end{aligned}
$$
since $C(X)$ has been embedded isometrically in the multiplier algebra of ${\sf C}^*(\Xi)$\,, by Lemma \ref{lemma2}.
\end{proof}

We can prove now Theorem \ref{main}, (i)\,.

\begin{proof}
Lemma \ref{apatra} (with $A=\mathcal Q$) allows the identification 
$$
{\sf C}^*\big(\Xi_{\mathcal Q}\big)^\bullet\cong{\sf C}^*(\Xi)^\bullet/{\sf C}^*\big(\Xi_{X\setminus\mathcal Q}\big)\,,
$$
implying the fact that the spectrum ${\sf sp}_{{\sf C}^*(\Xi_{X\setminus\mathcal Q})}(F)$ of the image of $F$ in the quotient $C^*$-algebra ${\sf C}^*(\Xi)^\bullet/{\sf C}^*\big(\Xi_{X\setminus\mathcal Q}\big)$ coincides with ${\sf sp}\big(F_\mathcal Q\!\mid\!{\sf C}^*(\Xi_\mathcal Q)^\bullet\big)$\,. Then, since we assumed that $\,\supp(\kappa)\cap{\sf sp}\big(F_\mathcal Q\big)=\emptyset$\,, one has $\kappa(F)\in{\sf C}^*\big(\Xi_{X\setminus\mathcal Q}\big)$; this follows from Lemma \ref{lemma1}, with $\mathscr A\!:={\sf C}^*(\Xi)^\bullet$,  $\mathscr K\!:={\sf C}^*\big(\Xi_{X\setminus\mathcal Q}\big)$\,.

\smallskip
Using Lemma \ref{lemma2} and the fact that morphisms commute with the functional calculus, for any $\psi\in C(X)$ one has
$$
\begin{aligned}
\p\!\psi|_M\,\kappa(H_0)\!\p_{\mathbb B(\H_0)}\,&=\,\p\!\Pi^\bullet_0(\psi)\,\kappa\big[\Pi^\bullet_0(F)\big]\!\p_{\mathbb B(\H_0)}\\
&=\,\p\!\Pi^\bullet_0\big[\psi\cdot\kappa(F)\big]\!\p_{\mathbb B(\H_0)}\\
&=\,\p\!\psi\cdot\kappa(F)\!\p_{{\sf C}^*(\Xi)}.
\end{aligned}
$$

Let us fix $\epsilon>0$\,. By Lemma \ref{lemma3} with $A=\mathcal Q$ and $f=\kappa(F)$\,, there is a continuous function $\psi:X\to[0,2]$ with $\psi|_\mathcal Q=2$ and
$$
\p\!\psi|_M\,\kappa(H_0)\!\p_{\mathbb B(\H_0)}\,=\,\p\!\psi\cdot\kappa(F)\!\p_{{\sf C}^*(\Xi)}\,\le\epsilon\,.
$$
Let us set $W:=\psi^{-1}(1,\infty)$\,; it is an open neighborhood of $\mathcal Q$ on which ${\bf 1}_W\le\psi$\,. In particular, ${\bf 1}_{W_0}={\bf 1}_W|_M\le\psi|_M$\,. Then
$$
\begin{aligned}
\p\!{\bf 1}_{W_0}\,\kappa(H_0)\!\p_{\mathbb B(\H_0)}&=\,\p\!\kappa(H_0){\bf 1}_{W_0}\kappa(H_0)\!\p^{1/2}_{\mathbb B(\H_0)}\\
&\le\,\p\!\kappa(H_0)(\psi|_M)^2\kappa(H_0)\!\p^{1/2}_{\mathbb B(\H_0)}\\
&=\,\p\!\psi|_M\,\kappa(H_0)\!\p_{\mathbb B(\H_0)}\,\le\epsilon
\end{aligned}
$$
and the proof is over.
\end{proof}

\begin{Remark}\label{oscilez}
{\rm When applying Theorem \ref{main} one migh want to take Remark \ref{lateriu} into consideration, in order to have a rather large class of elements to which the result applies and the restriction operation is explicit. 
On the other hand, to $F\in{\sf C}^*(\Xi)^\bullet$ one can add a "potential" $V\in C(X)\subset{\sf C}^*(\Xi)^\bullet$, for which $\rho_\mathcal Q^\bullet(V)=V|_\mathcal Q$\,.}
\end{Remark}

\begin{Remark}\label{kinshasa}
{\rm One can show that ${\sf sp}\big(F_\mathcal Q\big)$ is included (very often strictly) in the essential spectrum of the operator $H_0$\,. Thus, in Theorem \ref{main}, the interesting case is 
$$
\supp(\kappa)\subset{\rm sp}_{\rm ess}(H_0)\!\setminus\!{\sf sp}\big(F_\mathcal Q\big)\subset{\rm sp}(H_0)\!\setminus\!{\sf sp}\big(F_\mathcal Q\big)\,.
$$
We sketch a proof of the stated spectral inclusion. Since $\mathcal Q$ is a closed invariant subset of $N$, it follows easily that ${\sf sp}\big(F_\mathcal Q\big)\subset{\sf sp}\big(F_N\big)$\,, hence it is enough to show that ${\rm sp}_{\rm ess}(H_0)={\sf sp}\big(F_N\big)$\,. We recall that the essential spectrum of an operator acting in a Hilbert space $\H$ coincides with the usual spectrum of its image in the Calkin $C^*$-algebra $\mathbb B(\H)/\mathbb K(\H)$\,, or in any of its $C^*$-subalgebra containing this image. But $\Pi_0$ sends ${\sf C}^*(\Xi_M)$ isomorphically to $\mathbb K(\H_0)$ and the embedding 
$$
{\sf C}^*(\Xi_N)^\bullet\cong{\sf C}^*(\Xi)^\bullet/{\sf C}^*(\Xi_M)\hookrightarrow\mathbb B(\H)/\mathbb K(\H)
$$
leads easily to the result. Actually, with some extra work, one can show that ${\rm sp}_{\rm ess}(H_0)$ is the (closed) union of the spectra of {\it all} the restrictions $F_{\mathcal Q'}$ of $F$ to all the quasi-orbits $\mathcal Q'\subset N$.
}
\end{Remark}

\section{Examples}\label{cuforit}

\subsection{Standard groupoids with group bundles at the boundary}\label{gorofarat}

We recall that {\it a group bundle} is a groupoid for which the source and the range maps coincide; then the groupoid can be written as the disjoint union of its isotropy groups.

\smallskip
Assume now that $\Xi$ is a standard groupoid over the unit space $X=M\sqcup N$, with open dense orbit $M$ having trivial isotropy. Also assume that the reduction 
$$
\Xi_N\!=:\!\Si=\bigsqcup_{n\in N}\Si_{n}
$$ 
is a group bundle, where $\Si_{n}\!:=\Xi^n_n$ is the isotropy group in $n\in N$. We set 
$$
{\rm q}\!:={\rm d}_N={\rm r}_N:\Si\to N
$$ 
for the bundle map. Since $\Xi$ is standard, each $\Si_{n}$ is an amenable, second countable, Hausdorrff locally compact group.  Since, by assumption, $\Xi$ has a (chosen) right Haar measure, this is also true for the closed reduction $\Si$\,; for every $n\in N$ the fiber measure $\lambda_n$ is a right Haar measure on $\Si_{n}$\,.

\smallskip
The orbit structure is very simple: There is the main (open, dense) orbit $M$, and then all the points $n\in N$ form singleton orbits by themselves: $\mathcal O_n=\mathcal Q_n=\{n\}$\,. The Hilbert space $\H_n$ is the $L^2$-space of the group $\Si_n$ with respect to the Haar measure $\lambda_n$\,. The groupoid $C^*$-algebra ${\sf C}^*\big(\Si_{n}\big)$ corresponding to the quasi-orbit $\{n\}$ coincides now with the group $C^*$-algebra of $\Si_n$\,. 

\smallskip
Let uf fix a function $F$ belonging to $L^{\infty,1}_{\rm cont}(\Xi)\subset {\sf C}^*(\Xi)$\,, cf. Remark \ref{lateriu}. One gets invariant restrictions 
$$
F_N\!:=\rho_N(F)=F|_{\Si}\in L^{\infty,1}_{\rm cont}(\Si)\,,
$$ 
$$
F_{\{n\}}\!:=\rho_{\{n\}}(F)=F|_{\Si_n}\in L^{\infty,1}_{\rm cont}(\Si_n)=L^1(\Si_n)\cap C(\Si_n)\,.
$$
It is easy to see that the operators $H_n\!:=\Pi_n(F)$ given by regular representations are just convolution operators:
\begin{equation}\label{hashen}
[H_n(u)](a)=\int_{\Si_n}\!\!F|_{\Si_n}\!\big(ab^{-1}\big)u(b)d\lambda_n(b)\,,\quad u\in L^2(\Si_n;\lambda_n)\,,\ a\in \Si_n\,.
\end{equation}

\begin{Remark}\label{mantorc}
{\rm Assume that the isotropy group $\Si_n$ is Abelian, with Pontryagin dual group $\widehat\Si_n$\,. The Fourier transform implements an isomorphism between the group $C^*$-algebra ${\sf C}^*(\Si_n)$ and the function algebra $C_0\big(\widehat\Si_n\big)$ of all complex continuous functions on $\widehat\Si$ decaying at infinity. The operator $H_n$ of convolution by $F|_{\Si_n}$ is unitarily equivalent to the operator of multiplication by the Fourier transform $\widehat{F|_{\Si_n}}$ acting in $L^2\big(\widehat\Si_n;\widehat\lambda_n\big)$\,, where $\widehat\lambda_n$ is a Haar measure on the dual group, conveniently normalized. The spectrum of this operator is simply the closure of the range of the function ${\sf R}_n\!:=\widehat{F|_{\Si_n}}\big(\widehat\Si_n\big)$\,.}
\end{Remark}

One may write down the localization result of Section \ref{hoameni} for standard groupoids with group bundles at the boundary, but without further assumptions it is impossible to make explicit the neighborhoods $W$ of the singleton quasi-orbits $\{n\}$\,. We refer to the next subsection for a better situation.

\subsection{Standard groupoids over quotient-type compactifications}\label{goroforit}

We are going to construct $X$ as a compactification of $M$ under the following

\begin{Hypothesis}\label{ipotenuza}
{\rm Let $M=M^{\rm in}\sqcup M^{\rm out}$ be a second countable Hausdorff locally compact space with topology $\mathcal T(M)$\,, decomposed as the disjoint union between a compact component $M^{\rm in}$ and a non-compact one $M^{\rm out}$. Let ${\sf p}:M^{\rm out}\!\to N$ be a continuous surjection to a second countable Hausdorff compact space $N$, with topology $\mathcal T(N)$\,. We are also going to suppose that no fibre $M^{\rm out}_n:={\sf p}^{-1}(\{n\})$ is compact. }
\end{Hypothesis}

Let us also denote by $\mathcal K(M)$ the family of all the compact subsets of $M$. We need a convention about complements: if $L\subset M^{\rm out}$, we are going to write 
$$
\tilde L\!:=M^{\rm out}\!\setminus\!L\,,\quad L^c\!:=M\!\setminus\!L=M^{\rm in}\sqcup\tilde L\,.
$$
For $E\in\mathcal T(N)$ and $K\in\mathcal K(M)$ one sets
$$
A^M_{E,K}\!:={\sf p}^{-1}(E)\cap K^c\subset M^{\rm out}\subset M,
$$
$$
A_{E,K}\!:=A^M_{E,K}\sqcup E\subset X\!:=M\sqcup N.
$$
One has 
$$
A^M_{E,K}=A_{E,K}\cap M,\quad E=A_{E,K}\cap N,
$$ 
as well as
$$
A_{E_1,K_1}\cap A_{E_2,K_2}=A_{E_1\cap E_2,K_1\cup K_2}\,.
$$
Let us set 
$$
\mathcal A(X)\!:=\big\{A_{E,K}\,\big\vert\,E\in\mathcal T(N),\,K\in\mathcal K(M)\big\}\,,
$$ 
$$
\mathcal B(X)\!:=\mathcal T(M)\cup\mathcal A(X)\,.
$$ 
It is easy to check that $\mathcal B(X)$ is the base of a topology on $X$ (coarser than the disjoint union topology on $M\sqcup N$)\,, that we denote by $\mathcal T(X)$\,. It also follows easily that $M$ embeds homeomorphically as an open subset of $X$ and the topology $\mathcal T(X)$\,, restricted to $N$, coincides with $\mathcal T(N)$\,.
 
 \begin{Lemma}\label{ctificaishn}
 The topological space $(X,\mathcal T(X))$ is a compactification of $(M,\mathcal T(M))$\,.
 \end{Lemma}
 
 \begin{proof}
 To show that $X$ is compact, it is enough to extract a finite subcover from any of its open cover of the form
 $$
 \big\{O_\gamma\mid\gamma\in\Gamma\big\}\cup\big\{A_{E_\delta,K_\delta}\mid\delta\in\Delta\big\}\,,
 $$
 where $O_\gamma\in \mathcal T(M)$\,, $E_\delta\in\mathcal T(N)$ and $K_\delta\in\mathcal K(M)$\,. Since $N$ is compact, one has $N=\bigcup_{\delta\in\Delta_0}\!E_\delta$ for some finite subset $\Delta_0$ of $\Delta$\,. We show now that the complement of the (open) set $\bigcup_{\delta\in\Delta_0}\!A^M_{E_\delta,K_\delta}$ in $M$ is compact. One has
 $$
 \begin{aligned}
& \Big[\bigcup_{\delta\in\Delta_0}\!A^M_{E_\delta,K_\delta}\Big]^c=\bigcap_{\delta\in\Delta_0}\!\big[{\sf p}^{-1}(E_\delta)\cap K_\delta^{c}\big]^c=\bigcap_{\delta\in\Delta_0}\!\big[{\sf p}^{-1}(E_\delta)^c\cup K_\delta\big]\,.
 \end{aligned}
 $$ 
 A set of the form $\,\bigcap_{j=1}^m(R_j\cup S_j)$ can be written as $\,\bigcup\big(T_1\cap T_2\cap\dots\cap T_m\big)$\,, over all possible $T_k\in\{R_k,S_k\}$\,, $\forall\,k=1,\dots,m$\,. In our case, each time at least one of the sets in an intersection is some $K_\delta$\,, the intersection is already compact. There is also a contribution containing no set $K_\delta$\,:
 $$
 \begin{aligned}
 \bigcap_{\delta\in\Delta_0}\!{\sf p}^{-1}(E_\delta)^c&= \Big[\bigcup_{\delta\in\Delta_0}\!{\sf p}^{-1}(E_\delta)\Big]^c=\Big[{\sf p}^{-1}\Big(\bigcup_{\delta\in\Delta_0}\!E_\delta\Big)\Big]^c\\
 &=\big[{\sf p}^{-1}(N)\big]^c=(M^{\rm out})^c=M^{\rm in}.
 \end{aligned}
 $$
 So the complement of $\,\bigcup_{\delta\in\Delta_0}\!A^M_{E_\delta,K_\delta}$ in $M$ can be covered by a finite number of sets $O_\gamma$ and we are done.
 
 \smallskip
 The fact that $M$ is dense in $X$ is obvious, since any neighborhood of a point belonging to $N$ contains a set $A_{E,K}$\,, with $E\ne\emptyset$\,, that meets $M$, because our fibers $M_n^{\rm out}$ are not compact.
 \end{proof}
 
 We could call $(X,\mathcal T(X))$ {\it the ${\sf p}$-compactification of} $(M,\mathcal T(M))$\,. It is meant to generalize the radial compactification of a vector space $M$, for which $M^{\rm int}=\{0\}$ and $N$ is a sphere.
 
 \begin{Remark}\label{descripshn}
 {\rm A function $\phi:X\to\mathbb C$ is continuous with respect to the topology $\mathcal T(X)$ if and only if 
 \begin{itemize}
 \item 
 the restrictions $\phi|_M$ and $\phi|_N$ are continuous,
 \item 
 for every $n\in N$ and for every $\epsilon>0$\,, there exists $E\in\mathcal T(N)$ with $n\in E$ and $K\in\mathcal K(M)$ such that $|\phi(m)-\phi(n)|\le\epsilon$ if ${\sf p}(m)\in E$ and $m\notin K$. 
 \end{itemize}
 Let us set
 \begin{equation*}\label{intind}
 \begin{aligned}
& C_{\sf p}(M)\!:=\\
 &\big\{\varphi\in C(M)\,\big\vert\, \exists\,M^{\rm in}\!\subset K\in\mathcal K(M)\ \,{\rm s.\,t.}\,\  \varphi(m)=\varphi(m')\ \,{\rm if}\,\,m,m'\notin K,\ {\sf p}(m)={\sf p}(m')\big\}\,.
 \end{aligned}
 \end{equation*}
 It is a unital $^*$-algebra consisting of bounded continuous functions that are asymptotically constant along all the fibers $M_n^{\rm out}$. It is not closed. Denoting by $C_{\mathcal T(X)}(M)$ the $C^*$-algebra formed of restrictions to the dense subset $M$ of all the elements of $C(X)$\,, one has
 \begin{equation*}\label{nimic}
 C(X)\cong C_{\mathcal T(X)}(M)\supset C_0(M)+C_{\sf p}(M)\,.
 \end{equation*}}
 \end{Remark}
 
 
 \smallskip
 Assume now, in the same context, that $\Xi$ is a standard groupoid over 
 $$
 X=M\sqcup N=M^{\rm in}\sqcup M^{\rm out}\sqcup N
 $$ 
 such that the reduction $\Xi_N\!=:\!\Si=\bigsqcup_{n\in N}\Si_{n}$ is a group bundle, as in the previous subsection. Now the setting is rich enough to allow a transparent form of Theorem \ref{main}.
 
 \begin{Proposition}\label{amia}
Let $F\in L^{\infty,1}_{\rm cont}(\Xi)\subset{\sf C}^*(\Xi)$ be a normal element and $n\in N$. Let $\kappa\in C_0(\mathbb R)$ be a real function with support that does not intersect the spectrum of the twisted convolution operator $H_n$ given in \eqref{hashen}. 

\begin{enumerate}
\item[(i)]
For every $\epsilon>0$ there is a neighborhood $E$ of $n$ in $N$ and a compact subset $K$ of $M$ such that
\begin{equation*}\label{trasnaie}
\big\Vert\,{\bf 1}_{\{m\notin K\mid {\sf p}(m)\in E\}}\,\kappa(H_0)\,\big\Vert_{\mathbb B(\H_0)}\le\epsilon\,.
\end{equation*}
\item[(ii)]
Suppose that $F$ is self-adjoint. One also has
\begin{equation*}\label{trasnea}
\big\Vert\,{\bf 1}_{\{m\notin K\mid {\sf p}(m)\in E\}}\,e^{itH_0}\kappa(H_0)u\,\big\Vert_{\H_0}\le\epsilon\p\!u\!\p_{\H_0}
\end{equation*}
uniformly in $t\in\R$ and $u\in\H_0$
\end{enumerate}
 \end{Proposition}
 
 For the given point $n$\,, if $\Si_n$ happens to be Abelian, as we stated in Remark \ref{mantorc}, $H_n$ is unitarily equivalent to a multiplication operator on the Pontryagin dual of $\Si_n$ and its spectrum is the closure of a range.
 
 \bigskip
\textbf{Acknowledgements:} M. M. acknowledges support by the Fondecyt Project 1160359.


 \end{document}